\documentclass[11pt]{amsart}
\usepackage{amssymb}
\usepackage[russian]{babel}
\newcommand{\bea}{\begin{eqnarray}}
\newcommand{\eea}{\end{eqnarray}}
\newcommand{\ba}{\begin{array}}
\newcommand{\ea}{\end{array}}
\newcommand{\bc}{\begin{center}}
\newcommand{\ec}{\end{center}}
\newcommand{\ol}{\overline}
\newcommand{\be}{\begin{equation}}
\newcommand{\ee}{\end{equation}}

\newcommand{\dsf}{\displaystyle\frac}

\def\g{\gamma}
\def\G{\Gamma}

\def\l{\lambda}
\def\g{\gamma}
\def\r{\rho}
\def\t{\theta}

\def\Q{\mathbb{Q}}
\def\Z{\mathbb{Z}}
\def\N{\mathbb{N}}
\def\C{\mathbb{C}}

\newtheorem{thm}{╥хюЁхьр}[section]
\newtheorem{lem}[thm]{╦хььр}

\begin{document}
%\Large
%\Large

\title[]{╬┴ ╬─═╬╔ ╧╬╦╚═╬╠╚└╦═╬╔ $P$-└─╚╫┼╤╩╬╔ ─╚═└╠╚╫┼╤╩╬╔ ╤╚╤╥┼╠┼}

\author{╘.╠.╠єїрьхфют}
\address{╘.╠.╠єїрьхфют\\
╠хїрэшъю-ьрЄхьрЄшўхёъшщ Їръєы№ЄхЄ\\
═рЎшюэры№э√щ ╙эштхЁёшЄхЄ ╙чсхъёЄрэр\\
┬єчуюЁюфюъ, ╥р°ъхэЄ, 100174\\
╙чсхъшёЄрэ} \email{{\tt far75m@yandex.ru}, {\tt
farrukh\_m@iiu.edu.my}}

\author{╙.└.╨ючшъют}
\address{╙.└.╨ючшъют\\
╚эёЄшЄєЄр ьрЄхьрЄшъш ш шэЇюЁьрЎшюээ√ї Єхїэюыюушщ └═╨╙ч\\
єы. ╘.╒юфцрхтр,29, ╥р°ъхэЄ, 100125\\
╙чсхъшёЄрэ}\email{{\tt rozikovu@yandex.ru}}
\begin{abstract}
┬  ЁрсюЄх фы   $p$-рфшўхёъющ фшэрьшўхёъющ ёшёЄхь√
$f(x)=x^{2n+1}+ax^{n+1}$ эр ьэюцхёЄтх ъюьяыхъёэ√ї $p$-рфшўхёъшї
ўшыхё $\C_p$, яюыэюёЄ№■ юяшёрэ√ фшёъш ╟шухы  ш рЄЄЁръЄюЁ√ Єръшї
ёшёЄхь.
\end{abstract}

\maketitle

\footnotetext[1]{═рёЄю ∙шщ рфЁхё (╘.╠.): Department of Comput. \&
Theor. Sci., Faculty of Sciences, IIUM, P.O. Box, 141, 25710,
Kuantan, Pahang, Malaysia }

\section{┬тхфхэшх}

$p$-рфшўхёъшх ўшёыр тяхЁт√х с√ыш ттхфхэ√ эхьхЎъшь ьрЄхьрЄшъюь
╩.├хэчхыхь. ╧юёых юЄъЁ√Єш  $p$-рфшўхёъшї ўшёхы, юэш
ЁрёёьрЄЁштрышё№ ъръ ўшёЄю ьрЄхьрЄшўхёъшщ юс·хъЄ шёёых\-фютрэш .
═рўшэр  ё 1980-ї уюфют Ёрчышўэ√х ьюфхыш, юяшёрээ√х эр  ч√ъх
$p$-рфшўхёъюую рэрышчр, ръЄштэю шчєўр■Єё . ╨рчышўэ√х яЁшьхэхэш 
Єръшї ўшёхы ъ ЄхюЁхЄшўхёъющ Їшчшъх  с√ыш яЁхфыюцхэ√ т ЁрсюЄрї
\cite{1}-\cite{5}, ъ ътрэЄютющ ьхїрэшъх - т \cite{6},
 ьэюушь фЁєушь юсырёЄ ь Їшчшъш - т \cite{7},\cite{8}.

╚ёёыхфютрэш  т $p$-рфшўхёъющ ътрэЄютющ Їшчшъх ёЄшьєышЁютрыш
шчєўхэшх $p$-рфшўхёъшї фш\-эр\-ьш\-ўхёъшї ёшёЄхь (ёь., эряЁшьхЁ,
\cite{9}-\cite{12}). ═хъюЄюЁ√х °руш т ¤Єюь эряЁртыхэшш \cite{9}
яюърч√тр■Є, ўЄю фрцх яЁюёЄ√х (ьюэюьшры№э√х) фшёъЁхЄэ√х
фшэрьшўхёъшх ёшёЄхь√ $f(x)=x^n$ эрф яюы ьш $p$-рфшўхёъшї ўшёхы
$\Q_p$ ш $\C_p$ шьх■Є тяюыэх ъюьяыхъёэюх яютхфхэшх. ╥ръюх
яютхфхэшх ёє∙хёЄтхээю чртшёшЄ юЄ чэрўхэш  яЁюёЄюую ўшёыр $p$. ╤
шчьхэхэшхь $p$ рЄЄЁръЄюЁ√ юЄюсЁрцр■Єё  т фшёъш ╟шухы  ш юсЁрЄэю.
╫шёыю Ўшъыют ш шї фышэ√ Єръцх чртшё Є юЄ $p$ (ёь.\cite{13}). ┬ ёт чш ё ¤Єшь тючэшърхЄ
чрфрўр шчєўхэш  тючьє∙хээ√ї тшэрьшўхёъшї ёшёЄхь тшфр
$f_q(x) = x^n+q(x)$, уфх тючьє∙хэшх чрфрхЄё  эхъюЄюЁ√ь ьэюуюўыхэшь $q(x)$. ┬ ЁрсюЄрї
\cite{KhN},\cite{KhS} шёёыхфютрэ√ ёт чш Єръшї фшэрьшўхёъшї ёшёЄхь ё ьюэюьшры№э√ьш ёшёЄхьрьш. └ Єръцх
шчєўхэ√ яхЁшюфшўхёъшх Єюўъш тючьє∙хээшї ёшёЄхь. ▌Єш шёёыхфютрэш  яюърч√тр■Є, ўЄю
яютхфхэшх тючьє∙хээ√ї ёшёЄхь юЄышўэ√ юЄ юс√ўэюую (ьюэюьшры№эюую) ш яю¤Єюьє шчєўхэшх
Єръшї ёшёЄхь трцэю (ёь. Єръцх \cite{A,12}).┴юыхх юс∙шх шёёыхфютрэш  яюышэюьшры№э√ї фшэрьшўхёъшї ёшёЄхь
яЁютюфшышё№ т  \cite{23}-\cite{25}, уфх шчєўхэ√ ьэюцхёЄтр ─цєыш  (Julia) ш ╘рЄє
Єръшї ёшёЄхь. ▌Єш шёёыхфютрэш  ёЄшьєышЁє■Є шчєўхэш  ъюьяюэхэЄ
ьэюцхёЄтр ╘рЄє фы  яюышэюьшры№э√ї фшэрьшўхёъшї ёшёЄхь, ъюЄюЁюх
ёюфхЁцшЄ рЄЄЁръЄюЁ√ ш фшёъш ╟шухы  (ёь.
\cite{24},\cite{25},\cite{27}). ╧ю¤Єюьє, т \cite{MM} с√ыю
шчєўхэю рЄЄЁръЄюЁ√ ш фшёъш ╟шухы  сюыхх яЁюёЄющ фшэрьшўхёъющ ёшёЄхь√ тшфр
$f(x)=x^3+ax^2$ эрф $\Q_p$ яЁш тёхї тючьюцэ√ї чэрўхэш ї ярЁрьхЄЁр $a$.
┬ эрёЄю ∙хщ ЁрсюЄх сєфхЄ ЁрёёьрЄЁштрЄ№ё  юсюс°хэшх т√°х ёърчрээющ ёшёЄхь√с Є.х.
$g(x)=x^{2n+1}+ax^{n+1}$ эрф яюыхь $\C_p$. ╚чєўр■Єё  фшёъш
╟шухы  ш рЄЄЁръЄюЁ√ Єръющ ёшёЄхь√. ╟рьхЄшь, ўЄю т√сюЁ тшфр
юсєёыютыхэ Єхь, ўЄю эхяюфтшцэ√х Єюўъш ЇєэъЎшш $g(x)$ эрїюф Єё  т
 тэюь тшфх. ╟рьхЄшь, ўЄю  эр°х шёёыхфютрэшх ёє∙хёЄтхээю юёэютрэю
эр $p$-рфшўхёъюь рэрышчх.\\[2mm]

\section{╧ЁхфтрЁшЄхы№э√х ётхфхэш }

\subsection{$p$-рфшўхёъшх ўшёыр}

╧єёЄ№ $\Q_p$ яюых $p$-рфшўхёъшї ўшёхы, ъюЄюЁюх  ты хЄё 
яюяюыэхэшхь яюы  ЁрЎшюэры№э√ї ўшёхы $\Q$ яю юЄэю°хэш■
$p$-рфшўхёъющ эюЁь√, юяЁхфхыхээ√щ эр $\Q$, чфхё№ ш фрыхх $p$ -
ЇшъёшЁютрээюх яЁюёЄюх ўшёыю. ▌Єр эюЁьр юяЁхфхы хЄё  ёыхфє■∙шь
юсЁрчюь. ╩рцфюх ЁрЎшюэры№эюх ўшёыю $x\neq 0$ ьюцэю чряшёрЄ№ т тшфх
$x=p^r\dsf{n}{m}$, уфх $n$  ш $m$ эх фхы Єё  эр  $p$. ╥юуфр
$p$-рфшўхёър  эюЁьр $x$ Ёртэр $|x|_p=p^{-r}$.

▌Єр эюЁьр єфютыхЄтюЁ хЄ ёшы№эюьє эхЁртхэёЄтє ЄЁхєуюы№эшър:
$$
|x+y|_p\leq\max\{|x|_p,|y|_p\}.
$$
▌Єю ётющёЄтю яюърч√трхЄ  эхрЁїшьхфютюёЄ№ эюЁь√.

╚ч ¤Єюую ётющёЄтр эхяюёЁхфёЄтхээю ёыхфє■Є ёыхфє■∙шх:

1) хёыш  $|x|_p\neq |y|_p$, Єю $|x-y|_p=\max\{|x|_p,|y|_p\}$;

2) хёыш  $|x|_p=|y|_p$, Єю  $|x-y|_p\leq |x|_p$;

╧юых $\Q_p$ эх  ты хЄё  рыухсЁршўхёъш яюыэ√ь, яю¤Єюьє ўхЁхч
$\Q_p^a$ юсючэрўрхЄё  хх рыухсЁршўхёъюх чрь√ърэшх. ┬ ёшыє ЄхюЁхь√
╩Ёєыыр (ёь.[28, ЄхюЁхьр 14.1, 14.2]) эюЁьр чрфрээр  т $\Q_p$ шьххЄ
хфшэёЄтхээюх яЁюфюыцхэшх фю $\Q_p^a$, ъюЄюЁюх Єюцх  ты хЄё 
эхрЁїшьхфют√ь.  ╟рьхЄшь, ўЄю $\Q_p^a$ эх  ты хЄё  яюыэ√ь
юЄэюёшЄхы№эю ¤Єющ эюЁь√. ╧юяюыэхэшх $\Q_p^a$ Єръцх  ты хЄё 
рыухсЁршўхёъш яюыэ√ь (ёь. [28, ЄхюЁхьр 17.1]), ш юэю юсючэрўрхЄё 
ўхЁхч $\C_p$ ш эрч√трхЄё  {\it яюыхь ъюьяыхъёэ√ї $p$-рфшўхёъшї
ўшёхы}.

─ы  ы■сюую $a\in\C_p$ ш $r>0$ юсючэрўшь
$$
 U_r(a)=\{x\in\C_p :
|x-a|_p< r\}, \ \ \ S_r(a)=\{x\in\C_p : |x-a|_p= r\}.
$$

─юърчрЄхы№ёЄтю ёыхфє■∙шї ыхьь ьюцэю эрщЄш т
\cite{9},\cite{28}.\\[1.1mm]

\begin{lem}\label{2.1} ┼ёыш $a\in S_1(0)$, Єю $S_1(0)\setminus
U_1(a)\subset S_1(a)$.
\end{lem}

\begin{lem}\label{2.2} ╧єёЄ№ $C^k_n=n!/(k!(n-k)!), k\leq n$. ╥юуфр
$|C^k_n|_p\leq 1.$
\end{lem}

╬сючэрўшь $\G^{(m)}=\{ x\in\C_p : x^m=1\},\ m\in\N,$
$$
\G=\cup_{m=1}^{\infty}\G^{(m)}, \ \
\G_m=\cup_{j=1}^{\infty}\G^{(m^j)}, \ \ \G_u=\cup_{m:
(m,p)=1}^{\infty}\G_m.
$$

╫хЁхч $\t_{j,k} \ \ (j=\overline{1,k})$  юсючэрўшь $k$-Є√щ ъюЁхэ№
шч 1, яЁшўхь $\t_{1,k}=1$.\\[1.2mm]

\begin{lem}\label{2.3} ╤ыхфє■∙шх єЄтхЁцфхэш  тхЁэ√:
\begin{enumerate}
  \item ╧єёЄ№ $y^n=a$, уфх $a=\t_{j,n-1}$ фы  эхъюЄюЁюую
$j=\overline{1,n-1}$ ш $y\neq a$. ┼ёыш  $(n,p)=1$, Єю $y\in
S_1(a)$.

\item $\G_u\subset S_1(1)$;

\item $|C_{p^k}^j|_p\leq\frac{1}{p}$ {\it фы  ы■сюую}
$j=\overline{1,p^k-1}$;

\item  $\G_p\subset U_1(1)$.
\end{enumerate}
\end{lem}

╘єэъЎш   $f:U_r(a)\to\C_p$ эрч√трхЄё  {\it рэрышЄшўхёъющ }, хёыш
хх ьюцэю яЁхфёЄртшЄ№ т тшфх
$$
f(x)=\sum_{n=0}^{\infty}f_n(x-a)^n, \ \ \ f_n\in \C_p,$$ чфхё№
ёїюфшьюёЄ№ яюэшьрхЄё  т ёь√ёых эюЁь√ эр °рЁх $U_r(a)$. ┴юыхх
яюфЁюсэю юс рэрышЄшўхёъшї ЇєэъЎш ї ьюцэю эрщЄш т \cite{29}.

╬ёэют√  $p$-рфшўхёъюую рэрышчр ш  $p$-рфшўхёъющ ьрЄхьрЄшўхёъющ
Їшчшъш фрэ√  т \cite{7},\cite{17}.

\subsection{─шэрьшўхёъшх ёшёЄхь√ эр $\C_p$}

┬ ¤Єюь яєэъЄх ь√ эряюьэшь (ёь., эряЁшьхЁ, \cite{27},\cite{29})
эхъюЄюЁ√х шчтхёЄэ√х ЇръЄ√, юЄэюё ∙шхё  ъ фшэрьшўхёъющ ёшёЄхьх
$(f,U)$  эр $\C_p$, уфх  $f: x\in U\to f(x)\in U$ - рэрышЄшўхёър 
ЇєэъЎш  ш  $U=U_r(a)$ шыш $\C_p$.

 ╧єёЄ№ $f:U\to U$  ты хЄё  рэрышЄшўхёъющ ЇєэъЎшхщ. ╬сючэрўшь
$f^n(x)=\underbrace{f\circ\dots\circ f}_n(x)$, уфх $x\in U$. ┼ёыш
$f(x_0)=x_0$, Єюуфр $x_0$ эрч√трхЄё  {\it эхяюфтшцэющ Єюўъющ}.
═хяюфтшцэр  Єюўър $x_0$ эрч√трхЄё  {\it яЁшЄ уштр■∙хщ}, хёыш
ёє∙хёЄтєхЄ юъЁхёЄэюёЄ№ $U(x_0)$ Єюўъш $x_0$ Єрър , ўЄю фы  тёхї
$y\in U(x_0)$ шьххЄ ьхёЄю $\lim\limits_{n\to\infty}f^n(y)=x_0$.
┼ёыш $x_0$ - яЁшЄ уштр■∙р  Єюўър, Єюуфр {\it рЄЄЁръЄшЁє■∙шь
срёёхщэюь} эрч√трхЄё  ьэюцхёЄтю
$$
A(x_0)=\{x\in \C_p :\ f^n(x)\to x_0, \ n\to\infty\}.
$$

╧єёЄ№ $x_0$ - эхяюфтшцэр  Єюўър фы  ЇєэъЎшш $f(x)$. ├ютюЁ Є, ўЄю
°рЁ $U_r(x_0)$ (эрїюф ∙шщё  т $U$) сєфхЄ {\it фшёъюь ╟шухы }, хёыш
ърцфр  ёЇхЁр $S_{\r}(x_0)$, $\r<r$  ты хЄё  шэтрЁшрэЄэющ ёЇхЁющ
юЄэюёшЄхы№эю  $f(x)$, Є.х. хёыш $x\in S_{\r}(x_0)$, Єю тёх
шЄхЁрЎшш ¤Єющ Єюўъш эрїюф Єё  т Єющ цх ёрьющ ёЇхЁх, Є.х.
$f^n(x)\in S_{\r}(x_0)$ фы  тёхї $n=1,2\dots$. ╬с·хфшэхэшх тёхї
фшёъют ╟шухыр ё ЎхэЄЁюь т Єюўъх $x_0$ эрч√трхЄё  {\it ьръёшьры№э√ь
фшёъюь ╟шухы } ш юэю юсючэрўрхЄё  ъръ $SI(x_0)$.

{\bf ╟рьхўрэшх.}\cite{9} ┬ ъюьяыхъёэющ ухюьхЄЁшш ЎхэЄЁ фшёър
юфэючэрўэю юяЁхфхы хЄё  фшёъюь, ш Ёрчышўэ√х эхяюфтшцэ√х Єюўъш эх
ьюуєЄ шьхЄ№ юфшэ ш ЄюЄ цх фшёъ ╟шухы . ┬ эхрЁїшьхфютюь ёыєўрх
ЎхэЄЁ фшёър - ¤Єю ы■ср  Єюўър, яЁшэрфыхцр∙р  ¤Єюьє фшёъє. ╧ю¤Єюьє
т яЁшэЎшях Ёрчышўэ√ь эхяюфтшцэ√ь Єюўърь ьюцхЄ ёююЄтхЄёЄтютрЄ№ юфшэ
ш ЄюЄ цх фшёъ ╟шухы , ўЄю фхырхЄ эхрЁїшьхфют√щ ёыєўрщ юЄышўэ√ь юЄ
юс√ўэюую ёыєўр .

╧єёЄ№ $x_0$ - эхяюфтшцэр  Єюўър рэрышЄшўхёъющ ЇєэъЎшш $f(x)$.
╧юыюцшь
$$
\l=\frac{d}{dx}f(x_0).
$$

╥юўър  $x_0$ эрч√трхЄё  {\it рЄЄЁръЄшЁє■∙хщ}, хёыш $0\leq
|\l|_p<1$; {\it ёхфыютющ}, хёыш $|\l|_p=1$ ш {\it юЄЄрыъштр■∙хщ},
хёыш  $|\l|_p>1$.\\[1mm]

\begin{thm}\label{2.4}\cite{9} ╧єёЄ№  $x_0$ - эхяюфтшцэр  Єюўър
рэрышЄшўхёъющ ЇєэъЎшш $f:U\to U$. ╥юуфр ёыхфє■∙шх єЄтхЁцфхэш 
тхЁэ√:
\begin{enumerate}

\item [1]. хёыш $x_0$ рЄЄЁръЄшЁє■∙р  Єюўър фы   $f$, Єю юэр  ты хЄё 
яЁшЄ уштр■∙хщ  фы  фшэрьшўхёъющ ёшёЄхь√ $(f,U)$. ┼ёыш ўшёыю $r>0$
єфютыхЄтюЁ хЄ эхЁртхэёЄтє
$$
q=\max_{1\leq
n<\infty}\bigg|\frac{1}{n!}\frac{d^nf}{dx^n}(x_0)\bigg|_pr^{n-1}<1
\eqno(2.2)
$$
ш  $U_r(x_0)\subset U$, Єю $U_r(x_0)\subset A(x_0)$;

\item [2]. хёыш $x_0$ - ёхфыютр  Єюўър фы  $f$, Єю юэр  ты хЄё 
ЎхэЄЁюь фшёър ╟шухы . ┼ёыш ўшёыю $r>0$ єфютыхЄтюЁ хЄ эхЁртхэёЄтє
$$
s=\max_{2\leq
n<\infty}\bigg|\frac{1}{n!}\frac{d^nf}{dx^n}(x_0)\bigg|_pr^{n-1}<
1 \eqno(2.3)
$$
ш $U_r(x_0)\subset U$, Єю  $U_r(x_0)\subset SI(x_0)$;

\end{enumerate}

\end{thm}

\section{─шэрьшўхёър  ёшёЄхьр $ f(x)=x^{2n+1}+ax^{n+1} $}

┬ ¤Єюь яєэъЄх ЁрёёьюЄЁшь рэрышЄшўхёъє■ ЇєэъЎш■ $ f:\C_p\to \C_p$,
юяЁхфхыхээє■ яю ЇюЁьєых:
$$ f(x)=x^{2n+1}+ax^{n+1}, \eqno (3.1)$$
уфх $|a|_p<1$, $a\ne 0$, $n\in \N$.

─рыхх сєфхь яЁхфяюырурЄ№, ўЄю $p\geq 3$. ╚ёяюы№чє  ЄхюЁхьє 2.4
фюърцхь эхъюЄюЁ√х ётющёЄтр фы  фшэрьшўхёъющ ёшёЄхь√ (3.1).

╟рьхЄшь, ўЄю эхяюфтшцэ√ьш Єюўърьш ЇєэъЎшш (3.1)  ты ■Єё 
$\{x_i\}_{i=\overline{1,2n}}$, уфх $ x_i^n=c_+$, $ i=\ol{1,n}$ ш
$x_j^n=c_-$, $j=\ol{n+1,2n}$, чфхё№
$$
c_{\pm}=\frac{1}{2}(-a\pm\sqrt{a^2+4}).\eqno(3.2)
$$

─рыхх ўхЁхч $c$ сєфхь юсючэрўрЄ№ ышсю $c_-$, ышсю $c_+$.\\[1mm]

\begin{lem}\label{3.1} ╧Ёш $|a|_p<1$ ёяЁртхфыштю ЁртхэёЄтю
$|c_{\pm}|_p=1$. \end{lem}

\begin{proof} ┬ ёшыє $p\geq 3$, шьххь $|4|_p=1$, юЄъєфр шч ётющёЄтр 1) эюЁь√
$|\cdot|_p$ яюыєўшь эєцэюх ЁртхэёЄтю.  \end{proof}

╚ч ыхьь√ 3.1 т ърўхёЄтх ёыхфёЄтш  яюыєўшь

\begin{lem}\label{3.2} $ \{x_j\}_{j=1}^{2n}\subset S_1(0).$
\end{lem}

\begin{lem}\label{3.3}  ╧єёЄ№ $(2n+1,p)=1.$ ┼ёыш $f(y)=x_j$ фы 
эхъюЄюЁюую $j=\ol{1,2n}$ ш $y\ne x_j$, Єюуфр $y\in S_1(x_j).$
\end{lem}

\begin{proof} ╧Ёхфяюыюцшь, ўЄю $y\in U_1(x_j)$, Єюуфр
$y=x_j+\g,$ уфх $|\gamma|_p<1.$ ╤ыхфютрЄхы№эю, \bea
0&=&|y^{2n+1}+ay^{n+1}-x_j^{2n+1}-ax_j^{n+1}|_p\nonumber\\
&=&\bigg|\sum_{k=1}^{2n+1}C_{2n+1}^k\g^kx_j^{2n+1-k}+
a\sum_{l=0}^{n+1}C_{n+1}^l\g^l x_j^{n+1-l}\bigg|_p \nonumber \\
&=&|\g|_p|(2n+1)x_j^{2n}+\Delta_1|_p=|\g|_p,\nonumber \eea уфх ь√
шёяюы№чютрыш $(2n+1,p)=1$ ш $|\Delta_1|_p<1.$ ╬Єё■фр ёыхфєхЄ, ўЄю
$y\notin U_1(x_j).$ ╟эрўшЄ, $|y|_p\geq 1.$ ╬Єъєфр
$|y|_p^{2n+1}>|a|_p|y|_p^{n+1},$ Єръ ъръ $|a|_p<1$. ╚ч
$|f(y)|_p=|x_j|_p=1$ эрїюфшь, ўЄю $|y|_p=1$, ёыхфютрЄхы№эю, $y\in
S_1(0)\setminus U_1(x_j).$ ┬ ёшыє ыхьь√ 2.1 шьххь $y\in S_1(x_j).$
╦хььр фюърчрэр.
\end{proof}

\begin{thm}\label{3.4}
╤ыхфє■∙шх єЄтхЁцфхэш  тхЁэ√:
\begin{enumerate}

\item [(i)] ┼ёыш $(2n+1,p)=1$, Єюуфр $x_j$ - ЎхэЄЁ фшёър ╟шухы  ш
$SI(x_j)=U_1(x_j).$

\item [(ii)] ╧єёЄ№ $n=p^l,$ $l=1,2,...$ ╥юуфр
$SI(x_j)=U_1(c^{1/n})$ фы  ы■сюую $j=\ol{1,n}.$
\end{enumerate}
\end{thm}

\begin{proof} (i) ╟рьхЄшь, ўЄю
$$
|f'(x_j)|_p=|x_j^n|_p|(2n+1)x_j^n+a(n+1)|_p=1
$$
ш \bea
\bigg|\frac{1}{m!}f^{(m)}(x_j)\bigg|_p&=&\bigg|\frac{(2n+1)!}{m!(2n+1-m)!}x^{2n-m}_j+
\frac{(n+1)!}{(n+1-m)!m!}x_j^{n-m}\bigg|_p \nonumber \\
&\leq& \max\{|C^m_{2n+1}|_p,|C^m_{n+1}|_p\}\leq 1.\nonumber \eea
╧ЁютхЁшь єёыютшх ЄхюЁхь√ 2.4:
$$
q=\max_{1\leq m<\infty}\bigg|\frac{1}{m!}f^{(m)}(x_j)\bigg|_pr^{m-1}\leq
r^{m-1}\max\{|C^m_{2n+1}|_p,|C^m_{n+1}|_p\}\leq 1.
$$
▌Єю эхЁртхэёЄтю т√яюыэ хЄё , хёыш т√схЁхь $r<1$. ╥ръшь юсЁрчюь, т
ёшыє ЄхюЁхь√ 2.4, эхяюфтшцэр  Єюўър $x_j$   ты хЄё  ЎхэЄЁюь фшёър
╟шухы . ╤ыхфютрЄхы№эю, $U_r(x_j)\subset SI(x_j).$ ─юърцхь ЄхяхЁ№,
ўЄю $f(S_1(x_j))\ne S_1(x_j)$ фы  ы■сюую $j$. ╧єёЄ№ $y$ Єръюх, ўЄю
$f(y)=x_j$, Єюуфр т ёшыє ыхьь√ 3.3 шьххь $y\in S_1(x_j)$. ╩Ёюьх
Єюую, $x_j\notin S_1(x_j).$ ╤ыхфютрЄхы№эю, $f(y)\notin S_1(x_j).$
╙ЄтхЁцфхэшх (i) фюърчрэю.

(ii) ╧єёЄ№ $x^n_i=c$, Є.х. $x_i^{p^l}=c$. ╥юуфр
$\bigg(\dsf{x_i}{c^{1/n}}\bigg)^n=1$, Є.х. $\dsf{x_i}{c^{1/n}}\in
\Gamma_p.$ ┬ ёшыє ыхьь√ 2.3, шьххь $\Gamma_p\subset U_1(1).$
╤ыхфютрЄхы№эю, $x_i\in U_1(c^{1/n}), $ Єръ ъръ $|c^{1/n}|_p=1.$
▀ёэю, ўЄю  $x_i\in U_1(x_i).$ ╬Єъєфр $U_1(x_i)=U_1(c^{1/n}).$ ┬
ёшыє (i) яюыєўшь $SI(x_i)=U_1(c^{1/n}), i=\ol{1,n}.$ ╥хюЁхьр
фюърчрэр. \end{proof}

\begin{lem}\label{3.5} ┼ёыш $x^n=c, \ \ y^n=c$ ш $x\ne y $,
$(p,n)=1,$ Єю $|x-y|_p=1.$
\end{lem}

\begin{proof} ╧Ёхфяюыюцшь, ўЄю  $|x-y|_p<1$, Єю $x=y+\gamma$,\ \
$|\gamma|_p<1.$ ╚ёяюы№чє  $|c|_p=|y|_p=1$ ш
$|\sum^n_{k=2}C^k_n\gamma^{k-1}y^{n-k}|_p<1$ яюыєўшь \bea
0&=&|x^n-y^n|_p=\bigg|\sum^n_{k=1}C^k_n\gamma^{k}y^{n-k}\bigg|_p\nonumber \\
&=&|\gamma|_p \bigg|ny^{n-1}+\sum^n_{k=2}C^k_n\gamma^{k-1}y^{n-k}\bigg|_p=
|\gamma|_p.\nonumber \eea ╤ыхфютрЄхы№эю, $x=y$, ўЄю яЁюЄштюЁхўшЄ№
єёыютш■ ыхьь√. ╦хььр фюърчрэр. \end{proof}

\begin{lem}\label{3.6} ╧єёЄ№ $(2n+1,p)=1$ ш $(n,p)=1.$ ╥юуфр $
x_i\in S_1(c^{1/n})$ ш $ SI(x_i)\cap SI(x_j)=\emptyset ,\ \
i,j=\ol{1,n}, \ \ i\ne j.$
\end{lem}

\begin{proof} ╥ръ ъръ $x_i^n=c$, Єю $\dsf{x_i}{c^{1/n}}\in\Gamma_u.$ ╚ч ыхьь√
2.3 ёыхфєхЄ $\dsf{x_i}{c^{1/n}}\in S_1(1).$ ╬Єё■фр $ x_i\in
S_1(c^{1/n}),$ яюёъюы№ъє $ |c^{1/n}|_p=1.$ ┬ ёшыє ЄхюЁхь√ 3.4 (i),
$ SI(x_i)=U_1(x_i).$  ╧єёЄ№ $x_i\ne x_j$ ш $ y\in U_1(x_i)$, Є.х.
$ |y-x_i|_p<1.$ ╥юуфр, т ёшыє ыхьь√ 3.5, $|y-x_j|_p=1,$
ёыхфютрЄхы№эю, $y\notin U_1(x_j).$ ╥ръшь юсЁрчюь, $SI(x_i)\cap
SI(x_j)=\emptyset, \ \ i\ne j.$ ╦хььр фюърчрэр. \end{proof}

\begin{lem}\label{3.7} ╧єёЄ№ $(n,p)\ne 1, \ \ n=p^km, \ \ (m,p)=1$ ш $
c_{ij}=\xi_i\eta_j$, $ \xi_i\in \Gamma^{(m)}, \ \ \eta^{p^k}_j=c,
\ \ i=\ol{0,m-1}, \ \ j=\ol{0,p^k-1}.$  ╥юуфр
\begin{enumerate}
\item [(i)] $SI(c_{ij})=U_1(c_{ij})=SI(c_{il}), \ \
j,l=\ol{0,p^k-1}, j\ne l.$

\item [(ii)] $SI(c_{ij})\cap SI(c_{kl})=\emptyset, \ \ k\ne l$.
\end{enumerate}
\end{lem}

\begin{proof} (i) ┬ ¤Єюь ёыєўрх ыхуъю тшфхЄ№, ўЄю $(2n+1,p)=1.$ ╥юуфр,
ёюуырёэю ЄхюЁхьх 3.4, шьххь $SI(c_{ij})=U_1(c_{ij}).$ ╚ч єёыютш 
ыхьь√ яюыєўшь $\eta_i\in U_1(c^{1/p^k})$, юЄъєфр $\xi_i\eta_j\in
U_1(\xi_ic^{1/p^k}).$ ╤ фЁєующ ёЄюЁюэ√, $c_{ij}=\xi_i\eta_j \in
U_1(c_{ij}) $ ш ёыхфютрЄхы№эю, $U_1(c_{ij})=V_1(\xi_i c^{1/p^k}),
\ \ \forall j=\ol{1, p^k-1}.$

(ii) ┬ ёшыє єёыютш  ыхьь√ $ \xi_i\in \Gamma_u$, юЄё■фр шч ыхьь√
2.3 яюыєўшь $\xi_i \in S_1(1).$ ╟рьхЄшь, ўЄю $|\xi_i-\xi_j|_p=1$
(ёь. ыхььє 3.5) яЁш $i\ne j$.  ╧єёЄ№ $y\in U_1(c_{ij}),$  Єюуфр
$$ |y-c_{ik}|_p=|y-c_{ik}+\eta_k(\xi_i-\xi_j)|_p=1.$$ ▌Єю
ючэрўрхЄ, ўЄю $ y\not\in U_1(c_{jk}).$ ╤ыхфютрЄхы№эю, т ёшыє (i),
шьххь
$$SI(c_{ij})\cap SI(c_{kl})=\emptyset, \ \ k\ne l.$$
╦хььр фюърчрэр. \end{proof}

\begin{lem}\label{3.8} ┼ёыш $(2n+1, p)\ne 1, \ \ k\geq 1,$  Єюуфр эхяюфтшцэ√х Єюўъш
 ты ■Єё  рЄЄЁръ\-Єш\-Ёє\-■∙шьш. ┴юыхх Єюую $ U_1(x_j)\subset A(x_j)$ фы  ы■сюую
$j=\ol{1,2n}.$
\end{lem}

\begin{proof} ╧єёЄ№ $2n+1=p^km, \ \ k\geq 1 $ ё $(m,p)=1.$ ╥юуфр шьххь
$|f'(x_j)|_p<1.$ ╤ыхфютрЄхы№эю, ърцфр  эхяюфтшцэр  Єюўър  ты хЄё 
рЄЄЁръЄшЁє■∙шь. ╫Єюс√ юяЁхфхышЄ№ рЄЄЁръЄєЁє■∙шщ срёёхщэ ь√
шёяюы№чєхь єёыютшх (2.2) ЄхюЁхь√ 2.4, ъюЄюЁюх шьххЄ тшф
$$ q=\max_{1\leq n<\infty}\{|C^m_{2n+1}|_p,
|C^m_{n+1}|_p\}r^{n-1}<1.$$ ┼ёыш $r<1$,  ¤Єю єёыютшх т√яюыэ хЄё .
╥ръшь юсЁрчюь, $U_1(x_j)\subset A(x_j).$  \end{proof}

\begin{lem}\label{3.} ┼ёыш $(2n+1, p)\ne 1$,  Єю
$$ \bigcup_{\xi\in \Gamma_m}U_1(x_i\xi)\subset A(x_i).$$
\end{lem}

\begin{proof} ╧єёЄ№ $2n+1=p^km, \ \ (m,p)=1.$ ╨рёёьюЄЁшь $\xi \in
\Gamma_{m^k}, \ \ y\in U_1(x_i\xi).$ ╥юуфр $y=x_i\xi+\gamma $ ш $
|\gamma|_p<1.$ ╟рьхЄшь, ўЄю
$$
f^{k}(x+\gamma)=f(f^{k-1}(x+\gamma))=(x+\gamma)^{(2n+1)^k}
+\Delta_{\gamma}.
$$
┬ ёшыє $|a|_p<1$,  хёыш $|y|_p=1,$ Єю $|f^{k}(y)|_p=1$  фы 
$\forall k \in N.$  ╥ръшь юсЁрчюь, шч
$f^{k}(y)-(x_i\xi+\gamma)^{(2n+1)^k}=\Delta_{\gamma}$  ёыхфєхЄ
$|\Delta_{\gamma}|_p<1$  фы   $\forall |\gamma|_p<1.$ ╥хяхЁ№
ЁрёёьюЄЁшь \bea
|f^{(k)}(y)-x_i|_p&=&|f^{(k)}(y)-f^{(k)}(x_i)|_p\nonumber \\
&=&|(x_i\xi+\gamma)^{(2n+1)^k}+\Delta_{\gamma}-x^{(2n+1)^k}_i-\Delta_0|_p\nonumber\\
&=&\bigg|\gamma\sum^{(2n+1)^k}_{l=1}C^l_{(2n+1)^k}
\gamma^{l-1}(x_i\xi)^{(2n+1)^k-l}
+\Delta_{\gamma}-\Delta_0\bigg|_p<1.\nonumber \eea ╤ыхфютрЄхы№эю,
$f^{k}(y)\in U_1(x_i)\subset A(x_i).$ ╬Єъєфр $U_1(x_i\xi)\subset
A(x_i)$ ш
$$ \bigcup_{\xi \in \Gamma_m}U_1(x_i\xi)\subset A(x_i).$$
╦хььр фюърчрэр. \end{proof}

\begin{lem}\label{3.10} ┼ёыш  $x^n=c_+, \ \ y^n=c_- $, Єюуфр $|x-y|_p=1.$
\end{lem}

\begin{proof} ╚ч юяЁхфхыхэш  $c_\pm$ (ёь. (3.2)) эрїюфшь, ўЄю
$|c_+-c_-|_p=|\sqrt{a^2+4}|_p=1.$ ╤ыхфютрЄхы№эю,
$$ 1=|x^n-y^n|=|x-y|_p\bigg|\sum^{n-1}_{k=1}x^ky^{n-k}\bigg|_p\leq
|x-y|_p\leq 1$$ ╬Єъєфр $|x-y|_p=1.$ ╟фхё№ ь√ шёяюы№чютрыш, ўЄю
$|x|_p=|y|_p=|c_{\pm}|_p=1.$ \end{proof}

\begin{lem}\label{3.11} ╧єёЄ№ $p\geq 3$, $ (2n+1,p)=1$  ш $ (n,p)=1.$ Єюуфр
$$ SI(x_i)\cap SI(x_{n+j})=\emptyset, \ \  i,j=\ol{1,n}.$$
\end{lem}

\begin{proof} └эрыюушўэю фюърчрЄхы№ёЄтє ыхьь√ 3.6, эю чфхё№ шёяюы№чєхЄё  ыхььр
3.10 тьхёЄю ыхьь√ 3.5. \end{proof}

╬с·хфшэшт яюыєўхээ√х Ёхчєы№ЄрЄ√ ¤Єюую яєэъЄр ьюцэю ёЇюЁьєышЁютрЄ№
ёыхфє■∙є■ ЄхюЁхьє.

\begin{thm}\label{3.12}
╤ыхфє■∙шх єЄтхЁцфхэш  тхЁэ√:
\begin{enumerate}
  \item [(i)]  ┼ёыш $(2n+1, p)=1,$  Єюуфр $x_j$ - ЎхэЄЁ фшёър
╟шухы  ш $SI(x_j)=U_1(x_j), \ \ j=\ol{1,2n}.$
  \item[(ii)] ╧єёЄ№ $n=p^l, l\in \N.$ ╥юуфр $SI(x_j)=U_1(c_+^{1/n}), \ \
j=\ol{1,n}$ ш $SI(x_j)=U_1(c_-^{1/n}), \ \ j=\ol{n+1,2n}.$
   \item[(iii)] ╧єёЄ№ $(2n+1, p)=1$  ш $(n,p)=1.$  ╥юуфр  $x_i\in
S_1(c_+^{1/n}), \ \ i=\ol{1,n}$  ш $x_j\in S_1(c_-^{1/n}), \ \
j=\ol{n+1,2n}.$  ┴юыхх Єюую, $SI(x_i)\cap SI(x_j)=\emptyset$ фы 
ы■с√ї $i,j=\ol{1,2n}, \ \ i\ne j,$.
    \item [(iv)] ╧єёЄ№ $(n,p)\ne 1$, $n=p^km, \ \ (m,p)=1$ ш
$c_{ij}=\xi_i\eta_j,\ \ \xi_i\in \Gamma^{(m)}, \  \
\eta^{p^k}_j=c, \ \ i=\ol{0,m-1}, j=\ol{0,p^k-1}.$ ╥юуфр
$SI(c_{ij})=U_1(c_{ij})=SI(c_{il}) $  яЁш $l,j=\ol{0,p^k-1}, \ \
l\ne j$  ш $SI(c_{ij})\cap SI(c_{kl})=\emptyset, \ \ k\ne i.$
    \item[(v)] ┼ёыш $(2n+1,p)\ne 1$, Єю
$$\bigcup_{\xi \in \Gamma_m}U_1(x_i\xi)\subset A(x_i).$$
\end{enumerate}
\end{thm}

{\bf ┴ыруюфрЁэюёЄ№.} └тЄюЁ√ яЁшчэрЄхы№э√ яЁюЇхёёюЁрь
╚.┬.┬юыютшўє ш └.▐.╒Ёхээшъютє чр тэшьрэшх, яюыхчэ√х ёютхЄ√ ш
чрьхўрэш .

\end{document}